\documentclass[12 pt]{amsart}

\usepackage{amsthm}
\usepackage{amsmath}
\usepackage{amssymb}
\usepackage{latexsym}
\usepackage{url}
\usepackage{graphicx}
\usepackage{bmpsize}
\usepackage{mathtools}
\usepackage[margin=3cm]{geometry}
\DeclarePairedDelimiter{\ceil}{\lceil}{\rceil}

\author{Felix Goldberg}
\address{Caesarea-Rothschild Institute, University of Haifa, Haifa, Israel}
\email{felix.goldberg@gmail.com}

\author{Deepak Rajendraprasad}
\address{Caesarea-Rothschild Institute, University of Haifa, Haifa, Israel}
\email{deepakmail@gmail.com}

\author{Rogers Mathew}
\address{Caesarea-Rothschild Institute, University of Haifa, Haifa, Israel}
\email{rogersmathew@gmail.com}

\title{Domination in designs}
\date{May 14, 2014}

\newtheorem{thm}{Theorem}[section]
\newtheorem{cor}[thm]{Corollary}

\newtheorem{lem}[thm]{Lemma}
\newtheorem{prop}[thm]{Proposition}
\newtheorem{prob}[thm]{Problem}
\newtheorem{conj}[thm]{Conjecture}
\newtheorem{defin}[thm]{Definition}
\newtheorem{expl}[thm]{Example}

\newtheorem{qstn}[thm]{Question}

\begin{document}

\begin{abstract}
We commence the study of domination in the incidence graphs of combinatorial designs. Let $D$ be a combinatorial design and denote by $\gamma(D)$ the domination number of the incidence (Levy) graph of $D$. We obtain a number of results about the domination numbers of various kinds of designs. 

For instance, a finite projective plane of order $n$, which is a symmetric $(n^{2}+n+1,n+1,1)$-design, has $\gamma=2n$. 
We study at depth the domination numbers of Steiner systems and in particular of Steiner triple systems. We show that a $STS(v)$ has $\gamma \geq \frac{2}{3}v-1$ and also obtain a number of upper bounds. The tantalizing conjecture that all Steiner triple systems on $v$ vertices have the same domination number is proposed and is verified up to $v \leq 15$.

The structure of minimal dominating sets is also investigated, both for its own sake and as a tool in deriving lower bounds on $\gamma$. Finally, a number of open questions are proposed.\end{abstract}

\subjclass{05C69,05B05,51E15,51E10}

\keywords{block design, domination number, independent domination, finite projective plane, Steiner triple system}

\thanks{{This research was supported by the Israel Science Foundation (grant number 862/10.)}}

\maketitle

\section{Introduction}
In this note we commence the study of the domination properties of the incidence graphs of combinatorial designs. Rather surprisingly, while both graph domination and combinatorial designs are widely and deeply studied, there have been virtually no attempts to marry the two subjects. The qualified exceptions to this statement are discussed in Section \ref{sec:lit}. 


\section{Definitions and basic facts}
Let $G=(V,E)$ be a graph. A set of vertices $S \subseteq V$ is said to be \emph{dominating} if for every vertex $u \in V \setminus S$ there is an edge $e=(u,s) \in E$ whose other end, $s$, belongs to $S$. The \emph{domination number} $\gamma(G)$ is the minimum cardinality of such a set. 

Let $D$ be a family of $k$-subsets of $X=\{x_{1},x_{2},\ldots,x_{v}\}$. The family $D$ is called a \emph{$(v,k,\lambda)$-design over $X$} if, for every two distinct elements $x_{i},x_{j}$ of $X$, there are exactly $\lambda$ sets in $D$ that contain both $x_{i}$ and $x_{j}$. A design is called \emph{non-trivial} if $k<v$. We will assume $\lambda>0$ throughout, to avoid pathological cases.

The elements of $X$ are called the \emph{points} of $D$ while the sets in $D$ are called \emph{blocks} and their number is traditionally denoted by $b$. We shall always assume that a design $D$ has parameters $(v,k,\lambda)$ and block number $b$, unless otherwise explicitly stated. A design with $\lambda=1$ will be called a \emph{Steiner design}. 

A \emph{Steiner triple system} is a $(v,3,1)$-design. We denote the set of Steiner triple systems with $v$ points by $STS(v)$.


Observe that our definition of a design allows for repeated blocks so that the family $D$ of blocks is a multiset rather than a set. Designs which have no repeated blocks are called \emph{simple}. 

\begin{prop}
All Steiner designs are simple.
\end{prop}

\begin{defin}
Let $D$ be a design with point set $X$ and block set $\mathbb{B}$. The \emph{incidence graph} of $D$ is defined by $G_{D}=(X \cup \mathbb{B},E)$ where $(x,B) \in E \Leftrightarrow x \in B$.
\end{defin}

Informally speaking, $G_{D}$ is a bipartite graph whose two classes of vertices represent the points and blocks of $D$ and the adjacencies on $G_{D}$ encode the incidence structure of $D$.

We shall henceforth slightly abuse notation and
refer to $\gamma(G_{D})$ as the \emph{domination number} of $D$ and denote it simply by $\gamma(D)$.



\subsection{Related work}\label{sec:lit}

Laskar and Wallis in \cite{LasWal99} have obtained some results about the domination number of the line graph of  $G_{D}$, motivated by combinatorial chessboard problems. However, the domination number of $G_{D}$ itself was not considered by them.

The recent thesis of H\'{e}ger \cite{Heger_thesis} is probably closeset in spirit to the present work: H\'{e}ger constructs  dominating sets of a particular form in projective planes. 



\section{Some general results}

It is a well-known fact that every element of $X$ appears in the same number $r$ of blocks. This number is called the \emph{replication number} of $D$ and satisfies the equation $$r=\lambda\frac{v-1}{k-1}.$$

\begin{lem}\cite[p. 306]{GraphsDigraphs}\label{lem:walikar}
Let $G$ be a graph on $n$ vertices with maximum degree $\Delta$. Then $\gamma(G) \geq \ceil{\frac{n}{\Delta+1}}$.
\end{lem}

Lemma \ref{lem:walikar} immediately establishes a lower bound on $\gamma(D)$:

\begin{thm}\label{thm:naive}
$$\gamma(D) \geq \ceil[\Big]{\frac{v+b}{r+1}}.$$
\end{thm}

We can improve upon this bound by determining the \emph{fractional domination number} of $G_{D}$. Our argument applies in fact to every bipartite semiregular graph and is an easy variation on that of \cite[Theorem 7.4.1]{Fractional}.
\begin{thm}\label{thm:frac}
$$\gamma(D) \geq \gamma^{*}(D)=v\frac{r-1}{kr-1}+b\frac{k-1}{kr-1}.$$
\end{thm}
\begin{proof}
Define the function $f:V(G_{D}) \rightarrow \mathbb{R}$ by assigning the value $\frac{r-1}{kr-1}$ to the points of $D$ and the value $\frac{k-1}{kr-1}$ to the blocks of $D$. It is easy to see that for every vertex $v \in V(G_{D})$ it holds that
$$
\sum_{u \in N[v]}{f(u)}=1.
$$
Thus, $f$ is both a fractional dominating and a fractional packing function and by linear programming duality, it attains the value of $\gamma^{*}(G_{D})$.
\end{proof}

The bound of Theorem \ref{thm:frac} is always at least as strong as that of Theorem \ref{thm:naive}. The two bounds coincide for symmetric designs and for them only. The difference between the two bounds is roughly $v(\frac{1}{k}-\frac{1}{r})$ which may sometimes be quite large. 
\begin{expl}
There are two known \cite{McKaySta78} non-isomorphic $(126,6,1)$-designs. For these designs we have $b=525,r=25$ and our two lower bounds take the values $26$ and $38$, respectively.
\end{expl}

The \emph{transversal number} $\tau(H)$ of a hypergraph $H$ is the minimum cardinality of a set of vertices which intersects all edges of $H$ (for graphs this concept reduces to the familiar one of a vertex cover). See, \cite{Fur88} for more on this parameter. Our next two results flesh out useful connections between $\tau$ and $\gamma$.

\begin{thm}\label{thm:tau}
$$
\tau \leq \gamma \leq \frac{v+\tau}{2}.
$$
\end{thm}
\begin{proof}
In order to prove the lower bound let $S$ be a dominating set in $G_{D}$ of cardinality $\gamma=\gamma(D)$. Let $P$ denote the points in $S$ and $L$ the blocks in $S$. Now for every block $B \in L$ choose a point $p_{B} \in B$. Clearly, the set $P \cup \bigcup_{B \in L}{p_{B}}$ is a transversal of cardinality at most $|S|=\gamma$.

To prove the upper bound, let $X$ be a transversal of cardinality $\tau=\tau(D)$. For every two distinct points $x,y$ not in $X$ let $b_{xy}$ be a block which contains them. Clearly, all the points not in $X$ can be covered by at most $\frac{v-\tau}{2}$ blocks of the form $b_{xy}$. These blocks together with $X$ form a dominating set in $G_{D}$ with cardinality $\tau+\frac{v-\tau}{2}=\frac{v+\tau}{2}$.
\end{proof}

\begin{thm}\label{thm:trans2}
$$
\gamma \leq \tau +r.
$$
\end{thm}
\begin{proof}
Let $T$ be a minimum transversal of $D$ and let $x$ be any point. Denote by $B(x)$ the family of blocks which contain $x$. We claim that $T \cup B(x)$ is a dominating set in $G_{D}$. Indeed, all blocks are dominated by $T$ and if $y$ is a point, then there is some $1 \leq i \leq r$ such that $\{x,y\} \subseteq B_{i} \in B(x)$. 
\end{proof}

Another concept which is of use to us is that of an \emph{independent set} in $D$: a set of points $S$ is independent if no block of $D$ is entirely contained in $S$. Clearly, if $X$ is the ground set of $D$, then $X \setminus S$ is a transversal of $D$. The maximum cardinality of an independent set is denoted $\beta(D)$ and we have $$\beta(D)=v-\tau(D).$$ 

Therefore the previous theorem has a 
\begin{cor}\label{cor:beta}
$$\gamma \leq v-\frac{\beta}{2}.$$
\end{cor}

We can now use a result of Grable, Phelps and R{\"o}dl \cite{GraPheRod95} (in fact, the special case $t=2$ of it) to obtain an upper bound on $\gamma(D)$.
\begin{thm}\cite[Theorem 2.3]{GraPheRod95}\label{thm:grable}
For fixed $k \geq 2$ and $\lambda$ there is a constant $c^{'}$ such that for large enough $v$ every $(v,k,l)$-design $D$ satisfies
$$
\beta(D) \geq c^{'} \cdot v^{\frac{k-2}{k-1}}(\ln{v})^{\frac{1}{k-1}}.
$$
\end{thm}

\begin{thm}\label{thm:upper}
$$\gamma(D) \leq v-O(v^{\frac{k-2}{k-1}}(\ln{v})^{\frac{1}{k-1}}).$$
\end{thm}
\begin{proof}
Apply Corollary \ref{cor:beta} and Theorem \ref{thm:grable}.
\end{proof}

The result of Theorem \ref{thm:upper} is probably close to the true state of things for designs with small block size. Indeed, for Steiner triple systems $(k=3)$ it is best possible, as discussed in Section \ref{sec:steiner}. On the other hand, for symmetric designs better bounds can be obtained from other considerations, as discussed in Section \ref{sec:symmetric}.






\section{The structure of dominating sets in $G_{D}$}
We now turn to analyze in more depth the structure of dominating sets in $G_{D}$. Special emphasis will be placed on Steiner designs.

For a subset $S \subseteq V(G_{D})$ we define $\pi(S)$ to be the set of points of $D$ which correspond to vertices in $S$. Given a set of points $P$ of $D$ we observe that the blocks are naturally partitioned into two sets - those which intersect $P$ and those which are disjoint from $P$. We shall denote these two sets by $L(P)$ and $\widehat{L}(P)$, respectively.

For example, the Fano plane consists of seven triples, listed here as rows:
$$
\left(\begin{array}{ccc}

1&2&3\\
1&4&5\\
1&6&7\\
2&4&6\\
2&5&7\\
3&4&7\\
3&5&6\\
\end{array}\right).
$$
If $P=\{2,5\}$ then $\widehat{L}(P)$ consists of the two blocks $\{1,6,7\}$ and $\{3,4,7\}$. 

The following observation is simple but extremely useful:
\begin{lem}\label{lem:hat}
Let $S$ be a dominating set in $G_{D}$ and $P=\pi(S)$. Then $\widehat{L}(P) \subseteq S$.
\end{lem}
\begin{proof}
The vertices corresponding to the blocks in $\widehat{L}(P)$ are not dominated by the vertices of $P$, by definition. Therefore, in order to be dominated they must be included in $S$ in their own right.
\end{proof}

Lemma \ref{lem:hat} prompts us to consider as possible candidates for minimal domination sets of a special kind to be defined now. If $P$ is a set of points in $D$, let
$$
I_{P}=P \cup \widehat{L}(P).
$$

\begin{defin}
Let $D$ be a design and let $S \subseteq V(G_{D})$ be a set. If $S=I_{P}$ for some set of points $P$ of $D$, we say that $S$ is a \emph{neat} set.
\end{defin}

In our previous example we have:
$$
I_{P}=\{2,5,167,347\}.
$$

The set of points obtained by deleting a single point from a block $B \in D$ will be called a \emph{punctured block}.

\begin{thm}\label{thm:noincid1}
Let $D$ be a Steiner design and let $P$ be a punctured block of $D$. Then $I_{P}$ is a dominating set of $D$.
\end{thm}
\begin{proof}
Suppose that $P=B-\{x\}$ for some block $B$ of $D$. Since the blocks of $L(P)$ are dominated by $P \subseteq I_{P}$ and the blocks of $\widehat{L}(P)$ are included in $I_{P}$, we only need to show that all the points in $X \setminus P$ are dominated by $I_{P}$. 

Let $y \neq x$ be a point in $X \setminus P$. Then $x$ and $y$ belong together to some block $C$. If we can show that $C \in \widehat{L}(P)$, then both $y$ and $x$ are dominated by $I_{P}$. And indeed, if there is some point $a \in C \cap P$, then we see that points $x$ and $a$ belong together to both $B$ and $C$, which implies $B=C$. However, this is a contradiction as we now have $y \in B$.
\end{proof}

So far in this section we have been developing a way to construct dominating sets, yielding upper bounds on $\gamma(D)$. Now we take the opposite tack and investigate the structure of a given minimal dominating set, with a view to deriving lower bounds on $\gamma(D)$. 

\begin{lem}\label{lem:struc}
Let $S$ be a dominating set in $G_{D}$ and $P=\pi(S)$. Then $$|S| \geq \ceil[\Big]{\frac{v+|P|(k-1)}{k}}.$$
\end{lem}
\begin{proof}
There are $v-|P|$ points in $D$ which do not belong to $P$. These points must be dominated by blocks and since each block includes $k$ points, we need at least $\frac{v-|P|}{k}$ blocks to dominate them.
\end{proof}

\begin{lem}\label{lem:disjoint}
Let $D$ be a Steiner design over the ground set $X$ and let $x \in X$. Suppose that $x$ belongs to the $r$ blocks $B_{1},\ldots,B_{r}$. For $1 \leq i \leq r$ let $C_{i}=B_{i}-\{x\}$. Then the sets $C_{1},\ldots,C_{r}$ partition $X-\{x\}$.
\end{lem}
\begin{proof}
The sets $C_{1},\ldots,C_{r}$ are disjoint because of $\lambda=1$. To see that they cover all of $X-\{x\}$ simply count:
$$
r(k-1)=v-1.
$$
\end{proof}

Suppose that $S \subseteq V(G)$ and let $v \in S$. A vertex $u \in V \setminus S$ so that $v$ is $u$'s only neighbor in $S$ will be called an \emph{external private neighbour} of $v$. 
\begin{lem}\cite[Proposition 6]{BolCock79}\label{lem:bolcock}
Let $G$ be a graph without isolated vertices. Then $G$ has a minimum dominating set $S$ in which every vertex has an external private neighbour.
\end{lem}

\begin{thm}\label{thm:bol}
Let $D$ be a Steiner design. Then $$\gamma(D) \geq \ceil[\Big]{\frac{2v}{k}}-1.$$
\end{thm}
\begin{proof}
Let $S$ be a minimum dominating set in $G_{D}$ as provided by Lemma \ref{lem:bolcock} and let $P=\pi(S)$. Suppose first that $P=X$. Since $\lambda=1$ we have that $|P|=|T|=v \geq r-1$. Therefore let us now assume that there is a point $x \notin P$. Since $S$ is dominating, there is some block $B \in S$ so that $x \in B$. 

The block $B$ has an external private neighbour $y \notin P$. Therefore all other $r-1$ blocks to which $y$ belongs are not in $S$. But by Lemma \ref{lem:hat} this means that each of these blocks contains a point from $P$. Lemma \ref{lem:disjoint} ensures that all these points are distinct and therefore we have established:
$$
|P| \geq r-1.
$$
The conclusion now follows by applying Lemma \ref{lem:struc} and keeping in mind that $r(k-1)=v-1$:
$$
|S| \geq \frac{v+|P|(k-1)}{k} \geq \frac{v+(r-1)(k-1)}{k}=\frac{2v}{k}-1.
$$
\end{proof}


\section{Finite projective planes}\label{sec:planes}
A finite projective plane of order $q$ is a $(q^2+q+1,q+1,1)$-design. In this case we have $r=q+1$ and $b=q^{2}+q+1$.
Note that as $\lambda=1$ the graph $G_{D}$ does not contain quadrangles and therefore its girth is at least six.

\begin{lem}\label{lem:bd}\cite{BriDut89}
Let $G$ be a connected graph of girth at least $6$ and with minimum degree $\delta$. Then $\gamma(G) \geq 2(\delta-1)$.
\end{lem}

\begin{thm}\label{thm:pp}
Let $\Pi$ be a finite projective plane of order $q$. Then $\gamma(\Pi)=2q$.
\end{thm}
\begin{proof}
Since the graph $G_{\Pi}$ is regular of degree $q+1$, the lower bound follows from Lemma \ref{lem:bd}. To obtain the upper bound, let $P$ be a punctured block, that is $P=B-\{x_{0}\}=\{x_{1},x_{2},\ldots,x_{q}\}$ for some block $B$. If we can show that $|\widehat{L}(P)|=q$ we will be done, by Theorem \ref{thm:noincid1}.

Let us count the number of lines in $L(P)$. First, there are $q+1$ lines incident with $x_{1}$, including $B$. For each of the other $(q-1)$ points in $P$, there are $q+1$ lines incident with it and exactly one of them is $B$ which has been already counted. No other double counting occurs because if a line $\ell$ contains both $x_{i}$ and $x_{j} (2 \leq i,j \leq q)$, then $\ell$ must be $B$ itself. Thus we have:
$$
|L(P)|=(q+1)+(q-1)q=q^{2}+1.
$$
Therefore $|\widehat{L}(P)|=q$.
\end{proof}

\section{Neat designs}\label{sec:struct}

\begin{defin}
\mbox{}
\begin{itemize}
\item
If $D$ has a neat dominating set of cardinality $\gamma(D)$, we say that $D$ is a \emph{neat design}.
\item
If all minimal dominating sets of $D$ are neat, we say that $D$ is a \emph{super-neat design}.
\end{itemize}
\end{defin}

Our proof of Theorem \ref{thm:pp} indicates that finite projective planes are neat. 

\begin{conj}
Finite projective planes are super-neat.
\end{conj}

The \emph{independent domination number} $i(G)$ of the graph $G$ is the minimum cardinality of an independent dominating set of $G$. It is clear that $\gamma(G) \leq i(G)$. We can now offer an alternative characterization of a neat design:

\begin{prop}
Let $D$ be a design. Then $D$ is neat if and only if $\gamma(G_{D})=i(G_{D})$.
\end{prop}

For a recent survey on $i(G)$ see \cite{GodHen13}. 


In the remainder of this section we exhibit a number of designs of varying degrees of neatness.

\begin{expl}[Super-neat designs]
Using a computer we have verified that the Fano plane and the affine plane of order $2$ (that is, the unique $(9,3,1)$-design) are super-neat.
\end{expl}

\begin{expl}[A neat design which is not super-neat]\label{expl:neat}
Let $D$ be the following $(8,4,3)$-design, drawn from \cite{SpenceDataNonSym}:
$$
D=
\left(\begin{array}{cccc}
1&2&3&4\\
1&2&3&5\\
1&2&6&7\\
1&3&6&8\\
1&4&5&6\\
1&4&7&8\\
1&5&7&8\\
2&3&7&8\\
2&4&5&7\\
2&4&6&8\\
2&5&6&8\\
3&4&5&8\\
3&4&6&7\\
3&5&6&7\\
\end{array}\right).
$$
We claim that $\gamma(D)=5$. This fact can be verified by brute force search but we can also give a nicer argument, based on our results so far. Observe first that $b=14$ and $r=7$. Now Theorem \ref{thm:frac} tells us that $\gamma(D) \geq 4$. Suppose that there is a dominating set $S$ of cardinality four. Then, by Lemma \ref{lem:struc} we see that $P=\pi(S)$ has at most two elements. 

We now show that this leads to contradiction. Indeed, if $|P|=0$, then $S$ consists of four blocks, but the other ten blocks of $D$ are not dominated. If $|P|=1$ then the single point in $S$ dominates $r=7$ blocks and at most three other blocks may be dominated as members of $S$. This still leaves at least four undominated blocks. Finally, if $|P|=2$ then the two points in $S$ dominate exactly $2r-\lambda=11$ blocks and at most two other blocks may be dominated as members of $S$. This still leaves one undominated block. Therefore $\gamma(D) \geq 5$.

At this stage we can begin to look for dominating sets of cardinality five. A computer search tells us that there $442$ such dominating sets of which only $14$ are neat. Here are specimens of the neat type:
$$
S_{1}=\{1,2,4,8,(3,5,6,7)\}=I_{\{1,2,4,8\}}
$$
and of the non-neat type:
$$
S_{2}=\{1,2,3,(1,2,6,7),(3,4,5,8)\}.
$$

\end{expl}

Observe now that from a given $(v,k,\lambda)$ design $D$ we can form a $(v,k,2\lambda)$-design $2D$ by including two copies of each block of $D$. We have $\gamma(2D) \geq \gamma(D)$ since a dominating set of $2D$, with possible duplicates removed, clearly dominates $D$.

\begin{expl}[A non-neat design]\label{expl:nonneat}
Let $D$ be the same design as in Example \ref{expl:neat}. We consider the $(8,4,6)$-design $2D$. The non-neat set $S_{2}$ of the previous example is dominating in $2D$ as well and therefore $\gamma(2D)=5$. However, any neat set in $2D$ must be of cardinality greater than $5$. 

Indeed, suppose that $S$ is such a set in $2D$. Unless it contains all eight points, it must contain some blocks. However, every block in $S$ appears twice, as there are two available copies of it. Therefore, we can remove the duplicates to obtain a smaller dominating set in $D$.
\end{expl}

We see from Example \ref{expl:nonneat} that multiple blocks may pose obstructions for neatness. Is this the only possible kind of obstruction? We conjecture so: 
\begin{conj}
Simple designs are neat.
\end{conj}

\section{Steiner triple systems}\label{sec:steiner}

It is well-known that an $STS(v)$ exists iff $v \equiv_{6} 1$ or $v \equiv_{6} 3 $. In this case we have 
$$b=\frac{v(v-1)}{6}, r=\frac{v-1}{2}.$$

\begin{conj}\label{conj:sts}
There exists a function $f:\mathbb{N}\rightarrow\mathbb{N}$ so that if $D \in STS(v)$, then $\gamma(D)=f(v)$.
\end{conj}

In other words, we conjecture that the domination number of a Steiner triple system does not depend on the structure of the system, but only on its order. 

What we know so far about the domination numbers of small Steiner triple systems can be summarized in a table:

\medskip

\begin{tabular}{|c|l|l|}

\hline

$v$\ & \# non-isomorphic $STS(v)$ \ & $\gamma$ \ \\ [0.5ex] 

\hline

7 & 1 & 4\\
9 & 1 & 5\\
13 & 2 & 9\\
15 & 80 & 10\\

\hline
\end{tabular}

\medskip


The computations for $v=15$ were performed by Gordon Royle \cite{RoyleSTS15}.

We now present an array of lower and upper bounds on $\gamma$ for Steiner triple systems.

\begin{thm}\label{thm:23}
Let $D \in STS(v)$. Then $$\gamma(D) \geq \ceil[\Big]{\frac{2v}{3}}-1.$$
\end{thm}
\begin{proof}
Take $k=3$ in Theorem \ref{thm:bol}.
\end{proof}

\begin{prop}
Let $D \in STS(19)$. Then $12 \leq \gamma(D) \leq 15$.
\end{prop}
\begin{proof}
The lower bound follows from Theorem \ref{thm:23}. The upper bound follows from Corollary \ref{cor:beta} and the fact that $\beta(D) \geq 7$ (cf. \cite[p. 306]{Triples}).
\end{proof}

Our next result is a negative one - we show that a linear upper bound is impossible. 
\begin{thm}\label{thm:linear}
Let $c \in (0,1)$ be a constant. Then for sufficiently large $v$ there is a $D \in STS(v)$ so that $\gamma(D) \geq cv$.
\end{thm}
\begin{proof}
By a result of Brown and Buhler \cite{BroBuh82} (cf. also \cite[p. 305]{Triples}) for a sufficiently large $v$ there exists a $D \in STS(v)$ with $\tau(D) \geq cv$. By Theorem \ref{thm:tau} the conclusion follows.
\end{proof}

\begin{thm}\label{thm:rodl}
There is a constant $C$ so that for all $v \geq v_{0}(C)$, if $D=STS(v)$ then $$\gamma(D) \leq v-C\sqrt{v\log{v}}.$$
\end{thm}
\begin{proof}
Put $k=3$ and $\lambda=1$ in Theorem \ref{thm:upper}.
\end{proof}

By the result of \cite{RodSin94} Theorem \ref{thm:rodl} is, in fact, best possible.

We can also record an upper bound which is not as sharp asymptotically as Theorem \ref{thm:rodl} but has the attractive advantage of an explicit expression.

\begin{thm}
Let $D \in STS(v)$. Then
$$
\gamma(D) \leq v-\sqrt{\frac{v}{2}} .
$$
\end{thm}
\begin{proof}
By a result of Erd{\"o}s and Hajnal \cite{ErdHaj66} (cf. also \cite[p. 305]{Triples}) we have that $\beta(D) \geq \sqrt{2v}$. The conclusion follows from Corollary \ref{cor:beta}.
\end{proof}

As a possible approach to Conjecture \ref{conj:sts} we pose 
\begin{qstn}
Let $D$ be an $STS(v)$ with a Pasch configuration. Let $D^{'}$ be an $STS(v)$ obtained from $D$ by a Pasch trade. Does $\gamma(D)=\gamma(D^{'})$ hold?
\end{qstn}

\section{Symmetric designs}\label{sec:symmetric}
The famous and basic inequality of Fisher (cf. \cite[p. 5]{Links}) asserts that
$$
b \geq v.
$$

Designs in which $b=v$ are called \emph{symmetric}. Note that $b=v$ is equivalent to $r=k$. Recall now that the projective planes are precisely those symmetric designs which have $\lambda=1$. If $\Pi$ is a finite projective plane, we can say that $\Pi$ is a symmetric $(k^{2}-k+1,k,1)$-design. Theorem \ref{thm:pp} thus asserts that $\gamma(\Pi)=2(k-1)$. We are now going to obtain a generalization of this fact for all symmetric designs.

\begin{lem}\cite[p. 5]{Links}\label{lem:sym}
Let $D$ be a symmetric $(v,k,\lambda)$-design. Every two blocks of $D$ intersect in exactly $\lambda$ points.
\end{lem}

Therefore, every block is a transversal and we have $\tau \leq k$ for symmetric designs. Using Theorem \ref{thm:trans2} we can then immediately deduce that $\gamma(D) \leq 2k$ but it is possible to obtain a stronger result as we shall now see. To derive it we shall need a slight strengthening of Lov\'{a}sz's generalized Helly theorem. Our next lemma is a special case of an argument given in \cite[13.25(b)]{Lovasz}. We reproduce the proof for the reader's convenience.
\begin{lem}\label{lem:helly}
Let $\mathcal{F}$ be a $k$-uniform intersecting family of sets. There is a distinct pair of sets $A_{1},A_{2} \in \mathcal{F}$ such that:
$$
|A_{1} \cap A_{2}| \leq k-(\tau(\mathcal{F})-1).
$$
\end{lem}
\begin{proof}
Choose $A_{1}$ arbitrarily. Clearly $|A_{1}| \geq \tau(\mathcal{F})$ by the intersection property. Choose a subset $Y \subseteq A_{1}$ such that $|Y|=\tau(\mathcal{F})-1$. As $Y$ is not a transversal, there is some set $A_{j}$ so that $Y \cap A_{j}=\emptyset$. Thus
\begin{equation}\label{eq:lov}
|A_{i} \cap A_{j}| \leq |A_{1}\setminus Y|=k-(\tau(\mathcal{F})-1).
\end{equation}
\end{proof}
\begin{lem}\label{lem:symdes}
Let $D$ be a symmetric $(v,k,\lambda)$-design. Then $$\tau(D)\leq k-\lambda+1.$$
\end{lem}
\begin{proof}
Choose blocks $A_{1},A_{2}$ satisfying \eqref{eq:lov}. We have:
$$
\lambda=|A_{1} \cap A_{2}| \leq k-(\tau(D)-1).
$$
\end{proof}

If $D$ is a symmetric design, then the \emph{dual design} $D^{T}$ is obtained by switching the roles of the points and blocks of $D$. $D^{T}$ is also a symmetric design with the same parameters (cf. \cite[p. 5]{Links}) thought it is not necessarily isomorphic to $D$.

\begin{thm}\label{thm:sym}
Let $D$ be a symmetric $(v,k,\lambda)$-design. Then $$\gamma(D)\leq 2(k-\lambda+1).$$
\end{thm}
\begin{proof}
By Lemma \ref{lem:symdes} there is a set of $k-\lambda+1$ points which dominates all blocks. On the other hand, applying Lemma \ref{lem:symdes} to $D^{T}$ we infer that there is a set of $k-\lambda+1$ blocks which dominates all points. The union of these two sets is a dominating set in $G_{D}$.
\end{proof}

\begin{cor}\label{cor:sym}
Let $D$ be a symmetric $(v,k,\lambda)$-design. Then $$\gamma(D) \leq 2(k-1).$$
\end{cor}
\begin{proof}
For $\lambda=1$ use Theorem \ref{thm:pp} and for $\lambda>1$ use Theorem \ref{thm:sym}.
\end{proof}

To conclude this section, we present a conjecture.
\begin{conj}\label{conj:biplane}
Let $D$ be a symmetric $(v,k,2)$-design with $k \geq 4$. Then $$\gamma(D)=k.$$
\end{conj}

\section{Some more questions for further research}



\begin{prob}
Understand the behaviour of $\gamma$  under the operation of forming a residual or a derived design.
\end{prob}

\begin{conj}
If $D_{1}$ is a residual design of $D$, then $\gamma(D_{1})=\gamma(D)-1$.
\end{conj}

\begin{qstn}
It is NP-Hard to compute $\gamma$ even for bipartite graphs (cf. \cite{Bert84}). But is there a polynomial algorithm for determining $\gamma(G_{D})$?
\end{qstn}


\section{Acknowledgments}
We wish to thank Professor Peter Dukes for an illuminating suggestion which has led us to Theorem \ref{thm:linear}, Professor Gordon Royle for sharing the calculations for $STS(15)$. We also thank Dr. Irith Hartman for a careful reading of a draft of the paper and Miss Miriam Farber for a discussion of Conjecture \ref{conj:biplane}.

\bibliographystyle{abbrv}
\bibliography{nuim}
\end{document}